\documentclass{article}
\usepackage{verbatim}
\newcommand{\R}{\mathbb{R}}
\newcommand{\Z}{\mathbb{Z}}
\newcommand{\C}{\mathbb{C}}
\usepackage{array}
\usepackage{natbib}
\usepackage{amsfonts}
\usepackage{bbm}
\usepackage{booktabs}
\usepackage{amsmath, mathrsfs}
\usepackage{mathtools}
\usepackage{enumitem}
\usepackage{amsthm}
\usepackage{xcolor}
\usepackage{graphicx} 
\graphicspath{ {./figures/} }
\usepackage{caption}
\usepackage{subcaption}
\newtheorem{theorem}{Theorem}[section]
\newtheorem{corollary}{Corollary}[theorem]
\newtheorem{lemma}[theorem]{Lemma}

\newtheorem{remark}[theorem]{Remark}
\newtheorem{proposition}[theorem]{Proposition}

\usepackage{hyperref}

\newcolumntype{C}[1]{>{\centering\let\newline\\\arraybackslash\hspace{0pt}}p{#1}}

\renewcommand{\mod}[1]{\ensuremath{\mathop{}[#1]}}

\DeclareMathOperator{\sinc}{sinc}
\DeclareMathOperator{\Spec}{Spec}

\newcommand{\set}[1]{\ensuremath{\left\{#1\right\}}}
\newcommand{\seq}[1]{\ensuremath{\left(#1\right)}}
\newcommand{\norm}[1]{\ensuremath{\left\lVert#1\right\rVert}}
\newcommand{\abs}[1]{\ensuremath{\left\lvert#1\right\rvert}}
\newcommand{\ip}[1]{\ensuremath{\left\langle#1\right\rangle}}

\newcommand{\frameset}{\ensuremath{\mathcal{F}}}

\newcommand{\defeq}{\vcentcolon=}

\usepackage{mathtools}
\usepackage{tikz}
\usetikzlibrary{matrix}

\usepackage[margin=1in]{geometry}

\title{On the structure of the Gram matrix for Gabor systems generated by B-splines.}

\date{\today}

\author{
Martin Buck\thanks{Department of Mathematics, Tufts University, Medford MA 02131, USA. Email: martin.buck@tufts.edu}
\and
Christina Frederick\thanks{Department of Mathematical Sciences, New Jersey Institute of Technology, Newark NJ 07102, USA. Email: christina.frederick@njit.edu}
\and
Kasso A.~Okoudjou\thanks{Department of Mathematics, Tufts University, Medford MA 02131, USA. Email: kasso.okoudjou@tufts.edu}
\and
Alexander Stangl\thanks{Department of Mathematics, New Jersey Institute of Technology, Newark NJ 07102, USA. Email: ajs282@njit.edu}
}

\begin{document}

\maketitle

\abstract{We consider the Gabor system $\mathcal{G}(g,a\mathbb{Z}\times b\mathbb{Z})$ generated by a continuous, compactly supported function $g$ over the time-frequency lattice generated by the parameters $a$ and $b$. We show that, under an appropriate ordering of the Gabor elements, certain submatrices of the Gram matrix of $\mathcal{G}(g,a\mathbb{Z}\times b\mathbb{Z})$ exhibit a block-Toeplitz structure. This structural property enables us to derive spectral results for finite sub-blocks of the Gram matrix  by appealing to the spectral theory of Toeplitz matrices. In particular, we apply our results to the Gram matrix of Gabor systems generated by the $N$th-order B-spline.

}

\section{Introduction}
\subsection{Background on the frame set problem}
Gabor analysis originated from two distinct perspectives: the operator-theoretic framework of J.~von Neumann \cite{MR66944} in 1932 and the communication-theoretic approach of D.~Gabor \cite{gabo1946} in 1946. Both authors independently investigated the completeness of the system 
\begin{equation}\label{gabgauss}
	\mathcal{G}(g,\mathbb{Z}^2) = \{g_{(n,k)}(\cdot) = g(\cdot - n) e^{2\pi ik \cdot} : (n,k) \in \mathbb{Z}^2\},
\end{equation}
utilizing the Gaussian $g(x)=e^{-\pi x^2}$ for its optimal localization in the time-frequency plane. This dual legacy highlights the fundamental nature of the integer lattice $\mathbb{Z}^2$ in tiling phase space (see \cite{heil2007history} for a detailed history).

More generally, the \textit{Gabor system} generated by a non-zero window function
\(g \in L^{2}(\R)\) along a (countable) discrete set of time-frequency shifts
\(\Gamma \subset \R^2\),
\begin{equation}\label{eq:Gabor_glambda}
	\mathcal{G}(g, \Gamma) = \{g_\gamma:=g(\cdot-\gamma_1) 
		e^{2\pi i\gamma_2\cdot}:\, \gamma=(\gamma_1,\gamma_2) \in \Gamma\},
\end{equation} is said to be a frame provided that there exist constants $0<\alpha\leq \beta < \infty$  such that for each \(f \in
L^2(\R)\)
\begin{equation}
	\alpha\norm{f}^2 \leq \sum_{\gamma \in \Gamma}
		\abs{\ip{f, g_{\gamma}}}^2 \leq \beta\norm{f}^2.
\end{equation}

It follows that there exists an associated set of functions 
\(\set{h_{\gamma}}_{\gamma \in \Gamma}\) which is also a Gabor frame for \(L^2(\R)\) for which the following reconstruction formulas hold for each \(f \in L^2(\R)\):
\begin{equation}
	f = \sum_{\gamma \in \Gamma} \ip{f, h_{\gamma}} g_{\gamma}
	  = \sum_{\gamma \in \Gamma} \ip{f, g_{\gamma}} h_{\gamma}.
\end{equation}
 When $\Gamma =a \Z \times b \Z$ is a separable lattice, it can be shown that the set of functions \(\set{h_{\gamma}}_{\gamma \in \Gamma}\) is a Gabor frame generated by a window function $h \in L^2(\R)$, and called a \textit{dual window} to $g$. In this case,  \(\set{h_{\gamma}}_{\gamma \in \Gamma}\) is a \textit{dual Gabor frame} to \(\set{g_{\gamma}}_{\gamma \in \Gamma} \)

 The full understanding of all window functions $g$, and all positive parameters $a, b$ such that $\mathcal{G}(g, a\mathbb{Z}\times b\mathbb{Z})$ is a Gabor frame for $L^2(\R)$, is still incomplete. Nonetheless, the last few decades of research have resulted in an impressive body of work on Gabor systems. For example, there exists a Nyquist-type condition dating back to the work of Daubechies~\cite{daubechies1990wavelet}, in which it is proved that when $ab >1$ is rational then $\mathcal{G}(g, a\mathbb{Z}\times b\mathbb{Z})$ is incomplete, for the irrational case we refer to \cite{MR1096368}. Furthermore, if $\mathcal{G}(g, a\mathbb{Z}\times b\mathbb{Z})$ is a Gabor frame for $L^2(\R)$ then $ab \leq 1$. The critical case $ab =1$ corresponds to the case when it is possible for $\mathcal{G}(g, a\mathbb{Z}\times b\mathbb{Z})$ to be an orthonormal or a Riesz basis. In such a case, the window function is necessarily poorly localized in the time-frequency plane as dictated by the Balian-Low Theorem \cite{MR1350699}. 
We refer to \cite{heil2007history} for an history of the density condition in Gabor analysis. We also refer to some of the foundational work of I.~Daubechies, who in the setting of  coherent state frames in the Bargmann-Fock conjectured that a Gabor system $G(g,a,b)$ on the rectangular lattice $a\Z\times b\Z$ is a frame if and only if $ab<1$ for the Gaussian, \cite{daubechies1988frames}.

Still, it is in general unknown how to characterize the so-called
\textit{full frame set} 
\(\frameset(g)\) of a given window \(g \in L^2(\R)\)
\begin{equation}
	\frameset(g) = \set{\;\Gamma \subset \R^2 : \mathcal{G}(g, \Gamma) =
		\set{g_{(\gamma_1, \gamma_2)}}_{(\gamma_1,\gamma_2) \in \Gamma}
		\text{ is a frame over } L^2(\R)\;}.
\end{equation}

Restricting
\(\Gamma\) to the lattice \(\Lambda = a\Z \times b\Z\) for positive
constants \(a,b\), we arrive at what is referred to as the
\textit{frame set problem}. Insights offered by this variant of the
full frame set problem are hoped to enable further investigation of the
more general problem.

Celebrated results in the early 90s characterized the frame set for the
Gaussian window as the open region $\{(a, b) \in \R^2_+: \, 0<ab<1\}$;  see \cite{seip1992density, seip1992density_bis, Lyubarskii92}. 
More recently, a similar result was obtained for classes of functions, including totally positive functions,  \cite{Grosto, Kloosto, grochenig2018sampling, grochenig2023totally}. For an overview of the state of this fascinating problem we refer to \cite{Olehon, grochenig2014mystery}. For some historical perspective, we refer to the work of Daubechies, Grossmann and Meyer, \cite{daubechies1986painless},   on painless non-orthogonal expansions from which one concludes that the rectangle $(0, 2)\times (0, 1/2)$ belongs to the frame set for many window functions $g$.

For $g\in L^2(\R), g\neq 0$, and $f\in L^2(\R)$, let $V_gf$ be the function defined on $\R^2$ by $$V_{g}f(x, \omega)=\int_{\R}f(t)\overline{g(t-x)}e^{-2\pi i t\omega}\, dt.$$ The modulation space $M^{1}(\R)$, also known as the Feichtinger algebra $S_0(\R)$, is the dense Banach subspace of $L^2(\R)$ consisting of all functions $g$ such that $V_gg\in L^1(\R^2)$ \cite{BenOko20, Groc2001}. It was proved in \cite{FeiKai04} that if the window function $g$ is in the modulation space $ M^1(\R)$  then its frameset $\mathcal{F}(g)$ is an open subset of $\{(a, b) \in \R^2_+: \, 0<ab<1\}$.

Examples of functions in $M^1(\mathbb{R})$ are the B-splines of order $N=2$ or higher, which are  even, compactly supported functions defined
recursively as
\begin{equation}\label{eq:sn}
	s_1(x) = \chi_{[-1/2, 1/2]}(x), \qquad s_N(x) = s_1(x) * s_{N-1}(x).
\end{equation}
They are of special interest to study for a number of reasons:
B-spline wavelets, which have a construction similar to Gabor
systems, have been enormously successful in finite element methods
\cite{xiang2006construction, jiawei2008new, xiang2006identification, wei2021b}. Gabor systems
corresponding to B-spline window functions have also shown promise
in acoustic scattering problems \cite{Kreuzer2019}.

The frame set for the first order B-spline $s_1$ was completely
characterized in \cite{dai2016abc}, and work on higher order
B-spline has gradually accumulated in the past decade (see
\cite{AtiKouOko1} \cite{atindehou2023frame} \cite{Olehon} \cite{Olehon1}
\cite{Ole4} \cite{ghosh2025gabor} \cite{ghosh2023obstructions} 
\cite{Grojan} \cite{Lemniel} \cite{Kloosto} for details). As noted above,
 \(\frameset(s_2)\) is an open set in $\{(a, b) \in \R^2_+: \, 0<ab<1\}$ \cite{FeiKai04},
but a full characterization of this set remains an open question. 


The authors of \cite{AtiKouOko1} introduced a framework for
determining the frame sets of compactly supported functions, including
the B-spline of order \(N \geq 2\), that unified many known results
on the frame sets of B-spline. Similarly, the authors of  \cite{atindehou2023frame} 
used a similar linear algebra based approach to shed new light
on the set \(\frameset(s_2)\) and enlarged the known frame set region by
deriving sufficient conditions for invertibility of the linear system
relating a frame with its dual. The key idea is based on the following result (see \cite[Theorem
2.10]{Ole1}): for \(g,h \in L^2(\R)\) and \(a,b > 0\), the Bessel
sequences \(\mathcal{G}(g, a\Z \times b\Z)\) and \(\mathcal{G}(h, a\Z \times
b\Z)\) are dual frames for \(L^2(\R)\) if and only if
\begin{equation}\label{eq:dualgeneral}
	\sum_{k \in \Z} \overline{g(x-\ell/b+ka)} h(x+ka) =
		b\delta_{\ell,0}, \forall \ell \in \Z \mbox{ and a.e. } x \in
		\left[-\frac{a}{2},\frac{a}{2}\right].
\end{equation}
The authors of \cite{atindehou2023frame} showed the existence of a
bounded compactly supported function \(h \in L^2(\R)\) that solves
\eqref{eq:dualgeneral} when \((a,b)\) belongs to a family of sets
parameterized by \(m\). To do this, the authors simplified and rewrote
\eqref{eq:dualgeneral} as a matrix-vector equation using the compact
support of \(s_2\) and derived sufficient conditions for invertibility
of the system.

\subsection{Contributions}

This paper aims to investigate the frameset of continuous, compactly supported functions. To this end, we initiate an analysis of the Gram matrix of the corresponding Gabor systems.
We build upon the structural investigations introduced in \cite{buckspectral}, and provide a spectral and structural analysis of the Gram matrices of Gabor systems generated by a class of continuous compactly supported windows, with a specific focus on B-splines of order $N$. In particular, we identify a unique structural hierarchy within these Gram matrices; we demonstrate in Theorem \ref{thm:offdiagTH} that certain a finite sub-Gramian  is a block-Toeplitz matrix where each constituent block admits a Hadamard factorization into a real-valued Toeplitz matrix and a rank-one Hankel matrix. Secondly, we derive an explicit analytic formulation for the Laurent symbols $t^{[\ell]}(x)$ associated with these blocks, proving in Theorem \ref{thm:tl_sinc} that they are governed by periodic sums of $\text{sinc}^{2N}$ kernels. Using this formulation, we establish an $O(\ell^{-N})$ spectral decay rate for the off-diagonal blocks, which provides the first formal mathematical justification for the block-diagonal dominance observed in the numerical analysis of spline-based Gabor systems. Finally, we leverage the theory of asymptotically equivalent circulant matrices to provide a pathway for estimating spectral bounds, specifically demonstrating how rational lattice parameters $a \in \mathbb{Q}$ induce the decay of the lower frame bounds in truncated systems.

The remainder of this paper is organized as follows. Section~\ref{sec:notations} sets the notations use throughout the paper. Section \ref{sec:gram_block} establishes the internal block structure of the Gram matrix and provides the proof for the Toeplitz-dot-Hankel factorization. Section \ref{sec:asymptotic} presents the spectral analysis of the block operators $\mathscr{G}^{[\ell]}$ via their associated Laurent symbols. We also introduce a sequence of asymptotically equivalent circulant matrices used for finite-section eigenvalue estimation. In addition, we discuss the implications of these spectral results on the frame set problem, specifically regarding the behavior of frame bounds on rational and irrational lattices. Finally, Section \ref{sec:conclusion} summarizes our findings and proposes future directions.

\section{Notation and preliminaries}\label{sec:notations}

Given a discrete collection of functions 
\(\set{g_\lambda}_{\lambda \in \Lambda}\) of \(L^2(\R)\), we define the following operators: 
\begin{itemize}
	\item The \textit{synthesis operator} \(\mathscr{T}\) is a mapping from
	      \(\C^{\Lambda}\) to \(L^2(\R)\) defined as
		  \begin{equation}
			  \mathscr{T}\seq{c_{\lambda}}_{\lambda \in \Lambda} = \sum_{\lambda
			      \in \Lambda} c_{\lambda} g_{\lambda}.
		  \end{equation}
	\item Its (formal) adjoint is the \textit{analysis operator} \(\mathscr{T}^*\) is given by
		  \begin{equation}
		      \mathscr{T}^*f = \seq{\ip{f, g_{\lambda}}}_{\lambda \in \Lambda}.
		  \end{equation}
	\item The \textit{frame operator} \(\mathscr{S}\) defined on
	      \(L^2(\R)\) is the composition of the synthesis and
		  analysis operators and given by  \(\mathscr{T}\mathscr{T}^*\)
		  \begin{equation}
		      \mathscr{S}f = \sum_{\lambda \in \Lambda} \ip{f, g_{\lambda}}
			      g_{\lambda}.
		  \end{equation}
	\item The \textit{Gram operator} \(\mathscr{G}\) is defined on 
	      \(\C^{\Lambda}\) as the composition of the analysis and
		  synthesis operators  \(\mathscr{T}^*\mathscr{T}\) and given  by
		  \begin{equation}
			  \mathscr{G}\seq{c_{\lambda}}_{\lambda \in \Lambda} =
			      \seq{\sum_{\lambda' \in \Lambda} c_{\lambda'}
				  \ip{g_{\lambda'}, g_{\lambda}}}_{\lambda \in \Lambda}.
		  \end{equation}
		  The matrix corresponding to \(\mathscr{G}\) is
		  referred to as the \textit{Gram matrix}. 
\end{itemize}
The frame operator and the Gram operator are related in that the two
share a spectrum, with the possible exception of zero, which may be
present in the spectrum of the Gram operator but not that of the frame
operator. For this reason, while the frame operator may be more readily
amenable to a finite sections method the presence of zero in the
spectrum of the Gram operator means that the spectra of the finite sections of the Gram
operator do not necessarily converge to the non-zero portion of the spectrum of the Gram operator (see \cite{christensen2005finite}).
This manifests as \textit{pollution} filling the spectral gap between
zero and \(\norm{S^{-1}}^{-1}\) when taking finite sections of \(\mathscr{G}\).

For the rest of the paper the discrete systems we consider are Gabor systems $\mathcal{G}(g, \Lambda)$ generated by a window function $g\in L^2(\R)$ and where $\Lambda$ is the lattice $a\Z \times b\Z$. We summarize in Table~\ref{tab:notation} below the notations to be used in the sequel. 

\begin{center}
	\renewcommand{\arraystretch}{1.3}
	\begin{tabular}{cl}
		\toprule
		\textbf{Symbol} & \textbf{Description} \\
		\midrule
		\(\mathcal{G}(g,S)\) & Gabor system \(\set{T_y M_{\omega} g :
			(y,\omega) \in S}\) with \(T_y f(x) = f(x-y)\), \(M_{\omega}f(x)
			 = e^{2\pi i\omega x} f(x)\) \\
		\(\Lambda\) & Infinite lattice \(a\Z \times b\Z\)\\
		\(\mathscr{G}\) & Infinite Gram matrix/operator over
			\(\Lambda\) \\
		\hline
		\(\Lambda_n\) & Finite shifted lattice \(\set{(ka,jb) :
			j,k = -\tfrac{n-1}{2}, \hdots, \tfrac{n-1}{2}}\) \\
		\(g_{j,k}\) & Atom \(g_{j,k}(x) \defeq T_{ka}M_{jb}g(x) =
			g(x-ka)e^{2\pi ijb x}\) \\
		\(G_n\) & Truncated \(n^2 \times n^2\) Gram matrix over
			\(\Lambda_n\) \\
		\(G_n^{[\ell]}\) & \(n \times n\) block at modulation difference
			\(\ell\) (\(\ell = j-j'\)), \(\abs{\ell} < n\) \\
		& Entries \((G_n^{[\ell]})_{k,k'} \defeq
			\ip{g_{j,k},g_{j-\ell,k'}}\), \(0 \leq k,k' \leq n-1\) \\
		\(\set{\lambda_m(C_n)}_{m=1}^n\) & Ordered eigenvalues of an
			\(n \times n\) matrix \(C_n\) \\
		\bottomrule
	\end{tabular}
	\vskip1em
	\par Table 1: Notation for time-frequency shifts, Gram matrices, and
	block structure.
	\label{tab:notation}
\end{center}

Fix \(a,b > 0\) with \(ab < 1\).  Let \(n \geq 3\) be an odd natural
number such that \(n > \frac{2}{a} + 1\), and consider  the
index set be \(\mathcal{I}_n \defeq \set{-\frac{n-1}{2}, \ldots,
\tfrac{n-1}{2}}\).

In the sequel we will focus on time-frequency shifts of the
$N$th order B-spline, \(g = s_N\), defined in \eqref{eq:sn}:
\begin{equation}
	g_{j,k}(x) = g(x - ka)e^{2\pi i jb x}, \qquad j,k \in \mathcal{I}_n
\end{equation}
in which the time-frequency shifts belong to a set of enumerated
elements,
\begin{equation*}
	\Lambda_n = \set{(ka, jb)}_{j,k \in \mathcal{I}_n}
\end{equation*}
The infinite Gram matrix corresponding to the Gabor system
\eqref{eq:Gabor_glambda} is defined entrywise by
\begin{equation}\label{eq:graminf_entry}
	(\mathscr{G})_{(j,k),(j',k')} \defeq \ip{g_{j,k}, g_{j',k'}} \qquad
		j, j', k, k' \in \Z.
\end{equation}

For each odd \(n \in \mathbb{N}\) with \(n \geq 3\) we introduce the truncated
system \(\set{g_{\lambda}}_{\lambda \in \Lambda_n}\). Then the sets
\(\Lambda_{2k+1}\) are nested:
\begin{equation*}
	\Lambda_3 \subset \Lambda_5 \subset \hdots, \qquad
		\bigcup_{k = 0}^\infty \Lambda_{2k+1} = \Lambda.
\end{equation*}

For the finite index set \(\mathcal{J}_n \defeq \mathcal{I}_n \times
\mathcal{I}_n\), let \(P_n : \ell^2(\Z^2) \to \ell^2(\mathcal{J}_n)\) denote
the coordinate projection onto \(\mathcal{J}_n\). Then the finite Gram
matrix \(G_n\) is the corresponding finite section
\begin{equation*}
	G_n = P_n \mathscr{G} P_n^*.
\end{equation*}
As \(n \to \infty\) the projections \(P_n\) converge strongly to the
identity, so \(G_n \to \mathscr{G}\) in the strong operator topology. The
entries of \(G_n\) are given by
\begin{equation}\label{eq:gram_entry}
	(G_n)_{(j,k), (j',k')} \defeq \ip{g_{j,k}, g_{j',k'}}, \qquad
		j, j', k, k' \in \mathcal{I}_n.
\end{equation}

We may view \(G_n\) as an \(n^2 \times n^2\) block matrix indexed by the
modulation indices \(j,j' \in \mathcal{I}_n\), where each block
\(G_n^{(j,j')} \in \C^{n \times n}\) corresponds to a fixed pair
\(j,j'\) and has entries
\begin{equation*}
	(G_n^{(j,j')})_{k,k'} = \ip{g_{j,k}, g_{j',k'}} \qquad
		k, k' \in \mathcal{I}_n.
\end{equation*}
This block structure will be exploited to describe the Toeplitz
properties and structural behavior of \(G_n\) and its infinite lattice
limit \(\mathscr{G}\).

Recall that a square \textit{Toeplitz matrix} \(A \in \C^{n \times n}\)
is a matrix whose entries obey the relation
\begin{equation*}
	A_{k,k'} = A_{k+1,k'+1} = a_{k-k'} \qquad
		k, k' = 0, 1, \ldots, n-1.
\end{equation*}

A \textit{Circulant matrix} is a Toeplitz matrix \(C \in \C^{n \times
n}\) having the form
\begin{equation*}
	C_{k,k'} = c_{k-k' \mod n} \qquad
		k, k' = 0, 1, \ldots, n-1.
\end{equation*}

A Hermitian matrix \(T\) has an \(n \times n\) \textit{block-Toeplitz}
structure with \(k \times k\) blocks if
\begin{equation*}
	T = \begin{pmatrix}
		A^{[0]} & A^{[1]} & \cdots & A^{[n-1]} \\[1ex]
		A^{[-1]} & A^{[0]} & \cdots & A^{[n-2]} \\[1ex]
		\vdots & \vdots & \ddots & \vdots \\[1ex]
		A^{[-(n-1)]} & A^{[-(n-2)]} & \hdots & A^{[0]}
	\end{pmatrix}
\end{equation*}
with
\begin{equation*}
	A^{[j]} \in \C^{k \times k}, \qquad A^{[j]} = (A^{[-j]})^*, \qquad
		j = -n+1, \hdots, n-1
\end{equation*}
If the matrices in the block-Toeplitz matrix are themselves Toeplitz,
then the matrix is \textit{Toeplitz-block-Toeplitz}.

\section{Toeplitz structures in the Gram matrix for $\mathcal{G}(s_N, a\mathbb{Z}\times b\mathbb{Z} )$}\label{sec:gram_block}

In this section we show first that the enumeration of $\Lambda_n$ can be done without loss of generality when considering the spectrum of ${G}_n$. This is because other orderings will be permutations of the rows and columns of ${G}_n$ that are simple similarity relations that leave the spectrum unchanged. Then, we show that $G_n$ has block structure with blocks that have an explicit representation as a Hadamard product of a Toeplitz and rank-one Hankel matrix. We then provide resulting properties of these matrices.



\begin{lemma} \label{lem:shift-invariance}
Let $s, t \in \mathbb{R}$ and define the shifted set:
\[
\widetilde{\Lambda_n} := \Lambda_n+(as, bt),
\]
and the corresponding Gabor atoms
\begin{equation}
\tilde{g}_{j,k}(x) = g(x - (ka+sa))e^{2\pi i (jb+tb) x}, \qquad j,k\in \mathcal{I}_n
\end{equation}

Then, for any $(j,k), (j',k') \in \Lambda_n$, the new Gram matrix $\tilde{G}_n$ has entries
\[
\langle \tilde{g}_{j,k}, \tilde{g}_{j',k'}\rangle = e^{2\pi i b as (j - j')} \langle g_{j,k}, g_{j',k'} \rangle.
\]
The shifting factor $s$ $(s=1/ab)$ can be chosen so that $\tilde{G}_n=G_n$ is unchanged. Furthermore, the spectra of $G_n$ and $\tilde{G}_n$ are identical.
\end{lemma}

\begin{proof}
The proof is a simple change-of-variables $u=x-as$ in the inner product seen in Equation \ref{eq:gram_entry}:
\begin{align*}
\langle \tilde{g}_{j,k}, \tilde{g}_{j',k'} \rangle
&= \int_{\mathbb{R}} g(x - ka - sa) e^{2\pi i (jb + tb)x} \cdot g(x - {k'}a - sa) e^{-2\pi i ({j'}b + tb)x} \, dx \\
&= \int_{\mathbb{R}} g(u - ka) e^{2\pi i jb (u + sa)} \cdot g(u - {k'}a) e^{-2\pi i {j'}b (u + sa)} \, du \\
&= e^{2\pi i b s a(j-j')} \int_{\mathbb{R}} g(u - ka) g(u - {k'}a) e^{2\pi i (jb - {j'}b) u} \, du \\
&= e^{2\pi i b s a(j-j')} \langle g_{j,k}, g_{j',k'} \rangle.
\end{align*}

Setting $U=\mathrm{diag}(e^{2\pi i a b s\, j})_{j=0}^{n-1}$,
\[
\tilde G_n=(U\otimes I_n)\,G_n\,(U\otimes I_n)^*,
\]
so $\tilde G_n$ is unitarily similar to $G_n$ and has the same spectrum.

\end{proof}

Having established that lattice shifts induce only phase modifications, we next analyze the internal block structure of $G_n$ itself. The below lemma shows that the finite Gram matrix $G_n$ is structured as an $n^2\times n^2$ block matrix whose $(j, j')$-th block is given by an $n\times n$ matrix $G_n^{[\ell]}$ with $\ell = j - j'$.

\begin{lemma} \label{lem: diag_blocks}
The matrix ${G}_n$ given by \eqref{eq:gram_entry} is block-Toeplitz with blocks that are banded for $n>\frac{2}{a}+1$. 
\end{lemma}

\begin{proof}

For $j, j', k, k'\in \mathcal{I}_n$, the $(j,j')-$th block in $G_n$
has its $(k,k')-$th entry given by
\begin{align}
\langle g_{j,k}, g_{j',k'} \rangle = \int_{\mathbb{R}} g(x - {ka}) g(x - {k'a}) e^{2\pi i ({jb} - {j'b})x} \, dx.
\end{align}
Due to the fixed modulation phase $e^{2 \pi i ({jb} - {j'b})x}$ appearing above, the blocks depend only on the modulation difference $\ell:=j-j'$, i.e. $G_n$ is
block–Toeplitz in $j$, and for $|\ell|\le n-1$ we write
    \begin{align*}
     (G_n^{[\ell]})_{k,k'}:=\langle g_{j,k}, g_{j-\ell,k'} \rangle = \int_{\mathbb{R}} g(x - {ka}) g(x - {k'a}) e^{2\pi i ({jb} - {(j-\ell)b})x} \, dx, 
    \end{align*}
This proves that $G_n$ is block-Toeplitz and has the form
\begin{align*}
        G_n &=
\begin{pmatrix}
G_n^{[0]} & G_n^{[1]} & G_n^{[2]} & \cdots & G_n^{[n-1]} \\
G_n^{[-1]} & G_n^{[0]} & G_n^{[1]} & \cdots & G_n^{[n-2]} \\
G_n^{[-2]} & G_n^{[-1]} & G_n^{[0]} & \cdots & G_n^{[n-3]} \\
\vdots & \vdots & \vdots & \ddots & \vdots \\
G_n^{[-(n-1)]} & G_n^{[-(n-2)]} & G_n^{[-(n-3)]} & \cdots & G_n^{[0]}
\end{pmatrix}.
\end{align*}
    
The compact support of $g=s_N$ on $[-\frac{N}{2},\frac{N}{2}]$ produces
\begin{align*}
    (G_n^{[\ell]})_{k,k'}&=
    \begin{cases}
e^{\pi i b l ({ka} + {k'a})}
\displaystyle\int_{{k'a} - 1}^{{ka} + 1} (x - {k'a} + 1)(-x + {ka} + 1) \cos(2\pi b\ell x) \, dx, & |{ka} - {k'a}| \leq N \\[2ex]
0, & \text{otherwise}
\end{cases}.
\end{align*}
Since the integral is zero when $|{k'a}-{ka}| > N$, or $|k-k'|>\frac{N}{a}$, the blocks ${G}_n^{[\ell]}$ are banded (recall that $n > \frac{N}{a}+1$).



\end{proof}

The next theorem describes the structure of the blocks given in Lemma \ref{lem: diag_blocks}

\begin{theorem} \label{thm:offdiagTH}
Let $G_n\in \mathbb{C}^{n^2\times n^2}$ be the truncated Gram matrix corresponding to a Gabor system generated by $g=s_N$ over $\Lambda_n$. Then each block $G_n^{[\ell]}\in \mathbb{C}^{n\times n}$, $|\ell|\leq n-1$ can be written as a Hadamard (entrywise) product:
\begin{align}
G_n^{[\ell]} = T_n^{[\ell]} \circ H_n^{[\ell]},\label{eqn:gram:toeplitz-dot-hankel}    
\end{align}

where $T_n^{[\ell]} \in \R^{n\times n}$ is a real-valued, symmetric Toeplitz matrix, and $H_n^{[\ell]}\in \C^{n\times n}$ is a rank-one Hankel matrix with unimodular entries.
\end{theorem}

\begin{proof}
For $k,k'\in \mathcal{I}_n$, we consider the $(k,k')-$entry of the block \(G_n^{[\ell]}\) corresponding to a fixed modulation difference $\ell$:
\[
(G_n^{[\ell]})_{k,k'}:=\langle g_{j,k}, g_{j-\ell,k'} \rangle = \int_{\mathbb{R}} g(x - {ka}) g(x - {k'a}) e^{2\pi i b \ell x} \, dx.
\]

After a symmetric change of variables \(x \to  x + \tfrac{1}{2}({ka} + {k'a})\), the entries are given by:
\begin{align*}
(G_n^{[\ell]})_{k,k'}
&= \int_{\mathbb{R}} g(x + \tfrac{1}{2}({ka} + {k'a}) - {ka})\, g(x + \tfrac{1}{2}({ka} + {k'a}) - {k'a})\, e^{2\pi i b \ell (x + \tfrac{1}{2}({ka} + {k'a}))} \, dx \\
&=e^{\pi i b\ell ({ka} + {k'a}) }\int_{\mathbb{R}} g\left(x - \tfrac{a}{2} (k-k')\right) g\left(x + \tfrac{a}{2} (k-k')\right)\, \cos(2\pi b\ell x) \, dx.
\end{align*}
Define the Toeplitz matrix $T_n^{[\ell]}$ with entries given by $(T_n^{[\ell]})_{k,k'}:= t^{[\ell]}_{\,k-k'}$, where, for $j=k-k'$
\begin{align}
    t^{[\ell]}_{\,j} &:= \int_{\mathbb{R}} g\left(x - \tfrac{a}{2} j\right) g\left(x + \tfrac{a}{2} j\right)\, \cos(2\pi b\ell x) \, dx, \label{eq:tjl}
\end{align}
and define $(H_n^{[\ell]})_{k,k'}:=e^{\pi i a b \ell (k+k')}$. Then $G_n^{[\ell]}=T_n^{[\ell]}\circ H_n^{[\ell]}$. Since $H_n^{[\ell]}=v^{[\ell]}(v^{[\ell]})^{\!\top}$ with
$v^{[\ell]}=\big(e^{\pi i a b \ell\,k}\big)_{k\in\mathcal{I}_n}$, it follows that $H_n^{[\ell]}$ is rank-one Hankel. Furthermore, since $g=s_N$ has support in $[-\frac{N}{2},\frac{N}{2}]$, we have $t^{[\ell]}_{\,j}=0$ for $|j|>\lfloor N/a\rfloor$, so
$T_n^{[\ell]}$ (hence $G_n^{[\ell]}$) is banded with bandwidth $\lfloor N/a\rfloor$.
\end{proof}

We can immediately relate the spectra of $G_n$ and that of $T_n^{[\ell]}$.

\begin{corollary}\label{cor:specgntnl}
Each block $G_n^{[\ell]}$ defined in Theorem \ref{thm:offdiagTH} shares the same singular values as the corresponding Toeplitz matrix $T_n^{[\ell]}$. Consequently, the Schatten norms and condition numbers of $G_n^{[\ell]}$ and $T_n^{[\ell]}$ are identical.
\end{corollary}

\begin{proof}
Let \( H_n^{[\ell]} =  v^{[\ell]}(v^{[\ell]})^{\!\top} \) and $T_n^{[\ell]} \in \mathbb{R}^{n\times n}$ be the Hankel and Toeplitz matrices given in the proof of Theorem~\ref{thm:offdiagTH}, where \( v^{[\ell]} \in \mathbb{C}^n \) is a vector of unit-modulus entries. Define the diagonal matrix \( D_n^{[\ell]} := \mathrm{diag}(v^{[\ell]}) \). Then the block \( G_n^{[\ell]} \) can be further written as
\[
G_n^{[\ell]} = D_n^{[\ell]} T_n^{[\ell]} D_n^{[\ell]}.
\]

Since $D_n^{[\ell]}$ is unitary,
\[
G_n^{[\ell]} G_n^{[\ell]*} = D_n^{[\ell]} T_n^{[\ell]} (T_n^{[\ell]})^* (D_n^{[\ell]})^*,
\]
which has the same eigenvalues as $T_n^{[\ell]} (T_n^{[\ell]})^*$. Thus the singular values coincide.
\end{proof}

Matrices with the decomposition in Theorem \ref{thm:offdiagTH} were coined {\it Toeplitz-dot-Hankel} matrices in \cite{townsend2018fast} and used to develop fast algorithms for conversions between coefficients in different orthogonal polynomial expansions. The top two plots in Figure \ref{fig: gram_block} show the magnitude and phase of the matrix $G_n$, and the bottom two plots show block structures that are revealed by the decomposition \eqref{eqn:gram:toeplitz-dot-hankel}: the Toeplitz matrices $T_n^{[\ell]}$ (bottom left) and the phase of the Hankel matrices $H_n$ (bottom right).

\begin{figure}[h!]
    \centering
  \includegraphics[width=\textwidth]{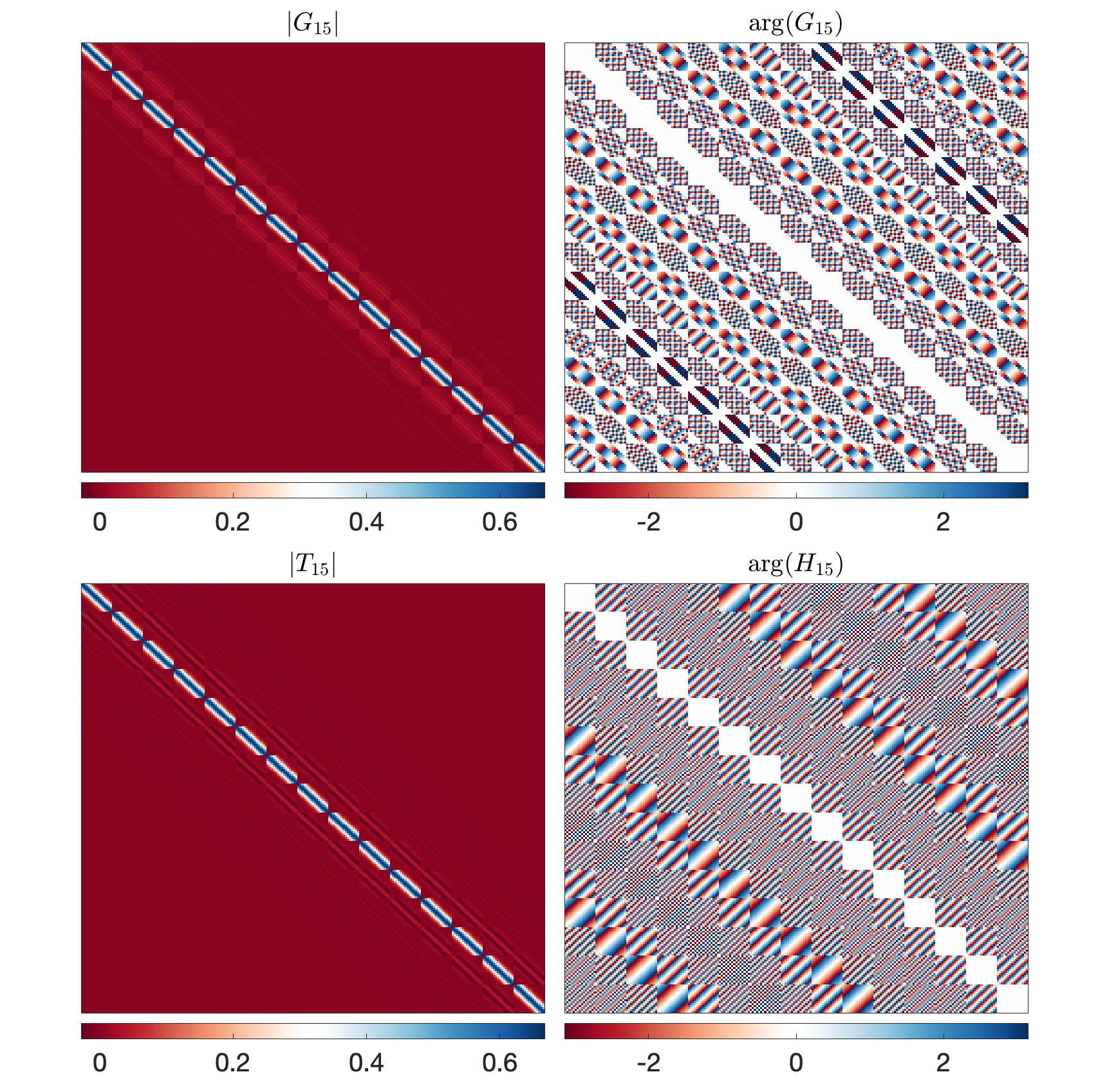}
     \caption{Top: The magnitude and phase of the finite Gram matrix $G_{15} \in \mathbb{C}^{115 \times 115}$ with blocks $G_{15}^{[\ell]} \in \mathbb{C}^{15 \times 15}$  Bottom: The magnitude of the matrix $T_{15}$ with Toeplitz blocks $T^{[\ell]}_{15}$ and the phase of the matrix $H_{15}$ with Hankel blocks $H_{15}^{[\ell]}$ in the decomposition \eqref{eqn:gram:toeplitz-dot-hankel}
     for $a=.25, b=1.5, n=15$.}
     \label{fig: gram_block}
\end{figure}

 By reordering indices according to modulation difference, $\mathscr{G}$ may also be expressed in block form with blocks $\mathscr{G}^{[\ell]}$ which allows us to extend Theorem~\ref{thm:offdiagTH} to $\mathscr{G}$.  For $\ell\in\mathbb{Z}$, the
$\ell$–block (grouping by modulation difference $j-j'=\ell$) is the
bi-infinite matrix on $\ell^2(\mathbb{Z})$ given by
\[
(\mathscr{G}^{[\ell]})_{k,k'} := \langle g_{j,k}, g_{j-\ell,k'}\rangle
= \int_{\mathbb{R}} g(x-ka)\,g(x-ka')\,e^{2\pi i b \ell x}\,dx.
\]

\begin{corollary}\label{cor:infBlocksTn}
Each bi-infinite block $\mathscr{G}^{[\ell]}$ admits a Hadamard (entrywise) factorization
\[
\mathscr{G}^{[\ell]} = T_{\infty}^{[\ell]} \circ H_{\infty}^{[\ell]},
\]
where $T_{\infty}^{[\ell]}$ is a banded, real-valued, symmetric Toeplitz operator on $\ell^2(\mathbb{Z})$
and $H_{\infty}^{[\ell]}$ is a complex-valued, rank-one Hankel operator with unimodular entries.
Moreover, letting $D_{\infty}^{[\ell]}=\mathrm{diag}(e^{\pi i a b \ell k})_{k\in\mathbb{Z}}$,
we have the diagonal unitary similarity
\[
\mathscr{G}^{[\ell]} = D_{\infty}^{[\ell]} \, T_{\infty}^{[\ell]} \, D_{\infty}^{[\ell]},
\]
so $\mathscr{G}^{[\ell]}$ and $T_{\infty}^{[\ell]}$ are unitarily equivalent and therefore share the same
spectra.
\end{corollary}

We recall that a matrix $A$ indexed by $\mathcal I\times\mathcal I$ is
\emph{per-Hermitian} if $A = J\,A^*\,J$, where $J$ is the reversal
(involution) $(Jx)_k := x_{-k}$ for symmetric indexing 
$k\in\mathcal I$ (and $(Jx)_k := x_{n-k+1}$ for $k\in\{1,\dots,n\}$).

\begin{proposition}\label{prop:perHermitian}
Let $g=s_N$ be the $N$th order B-spline. Then the finite Gram matrix $G_n$ and the infinite Gram operator $\mathscr{G}$
are per-Hermitian with respect to reversal in both indices:
\[
\overline{(G_n)_{(j,k),(j',k')}} \;=\; (G_n)_{(-j',-k'),(-j,-k)},
\qquad
\overline{(\mathscr{G})_{(j,k),(j',k')}} \;=\; (\mathscr{G})_{(-j',-k'),(-j,-k)}.
\]
Consequently, when $G_n$ is viewed as an $n\times n$ block matrix in the
modulation indices, each block $G_n^{[\ell]}$ is per-Hermitian:
\[
\overline{(G_n^{[\ell]})_{k,k'}} \;=\; (G_n^{[\ell]})_{-k',-k},
\qquad k,k'\in\mathcal I_n,
\]
and likewise for $\mathscr{G}^{[\ell]}$ on $\ell^2(\mathbb Z)$.
\end{proposition}

\section{Spectral analysis of sub-blocks of $G_n$ and $\mathscr{G}$}\label{sec:asymptotic}

We note that there is a Szeg\"{o} theorem for block Toeplitz matrices \cite{miranda2000asymptotic, gutierrez2008asymptotically, gutierrez2012block} that gives asymptotic results for the infinite block-Toeplitz matrix with blocks of size $k\times k$. Here, the sequence of block-Toeplitz matrices $G_n$ increases in both the number of blocks and the size $k$ of each block.

In this section we use the theory of Toeplitz matrices to describe the spectral properties of the block families
\[
T_n^{[\ell]}\in\mathbb{C}^{n\times n} \quad\text{and}\quad 
T_{\infty}^{[\ell]}\in\mathbb{C}^{\mathbb{Z}\times\mathbb{Z}},
\]
which govern the finite and infinite matrices $G_n^{[\ell]}$ and $\mathscr{G}^{[\ell]}$ through the  similarity relations proved in Section \ref{sec:gram_block}:
\[
G_n^{[\ell]} = D_n^{[\ell]} T_n^{[\ell]} D_n^{[\ell]}, 
\qquad 
\mathscr{G}^{[\ell]} = D_{\infty}^{[\ell]} T_{\infty}^{[\ell]} D_{\infty}^{[\ell]},
\]
where $D_n^{[\ell]}$ and $D_{\infty}^{[\ell]}$ are unitary diagonal matrices with unimodular entries, $(T_n^{[\ell]})_{k,k'}= t^{[\ell]}_{\,k-k'}$ for $k,k'\in \mathcal{I}_n$, and $(T_{\infty}^{[\ell]})_{k,k'}=t^{[\ell]}_{\,k-k'}$ for $k,k'\in \mathbb{Z}$.
Consequently, $G_n^{[\ell]}$ (resp. $\mathscr{G}^{[\ell]}$) and 
$T_n^{[\ell]}$ (resp. $T_{\infty}^{[\ell]}$) are unitarily equivalent and
share the same singular values (and the same spectra when the matrices are Hermitian).
Therefore, the asymptotic spectral behavior of $G_n^{[\ell]}$ (resp. $\mathscr{G}^{[\ell]}$)
is determined by the Toeplitz blocks $T_n^{[\ell]}$ (resp. $T_{\infty}^{[\ell]}$).

\subsection{Classical results}
We summarize classical results on Toeplitz matrices and refer to \cite{gray2006toeplitz, MR36936, zhu2017asymptotic, bottcher2000toeplitz} for comprehensive treatments.

Let \( \mathbf{S}^1 \coloneqq \{ z \in \mathbb{C} : |z| = 1 \} \) denote the complex unit circle. The \emph{Wiener algebra} \( W \) consists of functions \( f : \mathbf{S}^1 \to \mathbb{C} \) whose Fourier series converge absolutely:
\[
f(z) = \sum_{n=-\infty}^{\infty} \hat{f}_n z^n, \quad \|f\|_W = \sum_{n=-\infty}^{\infty} |\hat{f}_n| < \infty.
\]
The Fourier coefficients \( \hat{f}_n \) are given by
\[
\hat{f}_n = \frac{1}{2\pi} \int_0^{2\pi} f(e^{i\theta}) e^{-in\theta} \, d\theta.
\]
To each \( f \in W \) corresponds an infinite Toeplitz matrix \( T(f) \) defined by
\[
T(f) = \begin{bmatrix}
\hat{f}_0 & \hat{f}_{-1} & \hat{f}_{-2} & \cdots \\
\hat{f}_1 & \hat{f}_0 & \hat{f}_{-1} & \cdots \\
\hat{f}_2 & \hat{f}_1 & \hat{f}_0 & \cdots \\
\vdots & \vdots & \vdots & \ddots
\end{bmatrix}.
\]
The function \( f \), called the \emph{symbol} of the Toeplitz operator \( T(f) \), defines a bounded linear operator on \( \ell^p(\mathbb{Z}_+) \) for any \( 1 \leq p \leq \infty \), with operator norm satisfying \( \|T(f)\|_{\ell^p \to \ell^p} \leq \|f\|_W \) \cite{bottcher2005spectral}.


When the Fourier series of \( f \) has only finitely many non-zero coefficients, the associated Toeplitz matrix \( T(f) \) is a symmetric, banded bi-infinite matrix. In this case, \( f \) is called a \emph{Laurent polynomial}, and has the form
\[
f(z) = \sum_{j=-r}^{s} \hat{f}_j z^j, \quad z \in \mathbf{S}^1,
\]
for some integers \( r, s \geq 0 \) and real coefficients \( \hat{f}_j \in \mathbb{R} \).

In this case, \( T(f) \) has no eigenvalues, and its spectrum is the interval \( [\min f, \max f] \) \cite[Theorem~1.4]{bottcher2000toeplitz}.

\begin{theorem}[\cite{MR36936}]\label{thm:Tk_spec}
Let \( f \) be a real Laurent polynomial and let \( T(f) \) be the associated Toeplitz operator on \( \ell^2(\mathbb{Z}) \). Then
\[
\operatorname{sp}(T(f)) = \left[ \min_{z \in \mathbf{S}^1} f(z), \, \max_{z \in \mathbf{S}^1} f(z) \right].
\]
Moreover, if \( f \) is nonconstant, then \( T(f) \) has empty point spectrum (i.e., no eigenvalues).
\end{theorem}

\begin{theorem}[Lemma 4.1 in \cite{gray2006toeplitz} and Corollary 2.18 in \cite{bottcher2000toeplitz}]\label{cor-bdsfiniteteop}
Let \( T_n \) be the \( n \times n \) Hermitian Toeplitz matrix associated with the real-valued, continuous symbol \( t(x) \) defined on \( \mathbf{S}^1 \). Then the eigenvalues of \( T_n \) lie in the interval
\[
[\inf_{x} t(x), \sup_{x} t(x)],
\]
and satisfy the limits
\[
\lim_{n \to \infty} \lambda_1(T_n) = \inf_{x} t(x), \quad
\lim_{n \to \infty} \lambda_n(T_n) = \sup_{x} t(x).
\]
\end{theorem}

The following Szeg\"{o}-type limit theorem can be applied.
\begin{theorem}[Theorem 4.1, \cite{gray2006toeplitz}]\label{thm:Tn_lim_eigs}
Let \( T_n \) be a banded Toeplitz matrix with eigenvalues $\lambda_m(T_n)$ and real valued symbol $t$. Then for any function $F(x)$ continuous on $[\inf_xt(x), \sup_xt(x)]$, 

\[\lim_{n\to\infty} \frac{1}{n}\sum_{m=0}^{n-1}F(\lambda_m(T_n))=\frac{1}{2\pi}\int_{0}^{2\pi}F(t(\omega))d\omega. \]

\end{theorem}

\subsection{Spectral analysis of the infinite Toeplitz block $T_{\infty}^{[\ell]}$}

We fix a modulation difference $\ell$ in this section.

Let $T_{\infty}^{[\ell]}$ be the bi-infinite Toeplitz operator on
$\ell^2(\mathbb{Z})$ with entries $(T_{\infty}^{[\ell]})_{k,k'}=t^{[\ell]}_{\,k-k'}$,
where the coefficients $t^{[\ell]}_{\,j}$ are given by \eqref{eq:tjl}
and vanish for $|j|>\lfloor N/a\rfloor$. According to Corollary~\ref{cor:infBlocksTn}, $T_{\infty}^{[\ell]}$, is symmetric Toeplitz matrix with real entries. Its Laurent symbol is the real
trigonometric polynomial
\[
t^{[\ell]}(x)=\sum_{k=-\lfloor N/a\rfloor}^{\lfloor N/a\rfloor} t^{[\ell]}_{\,k}\,\cos(2\pi k x),
\qquad x\in[-\tfrac{1}{2},\tfrac{1}{2}].
\]
Figure \ref{fig:laurent_poly} shows plots of these Laurent polynomials for varying $a,b$ and $\ell$. In the case where the Toeplitz matrices are generated by an $N$th order B-spline, these Laurent polynomials have a closed analytic form.

\begin{figure} 
    \centering
   
         \includegraphics[width=\textwidth]{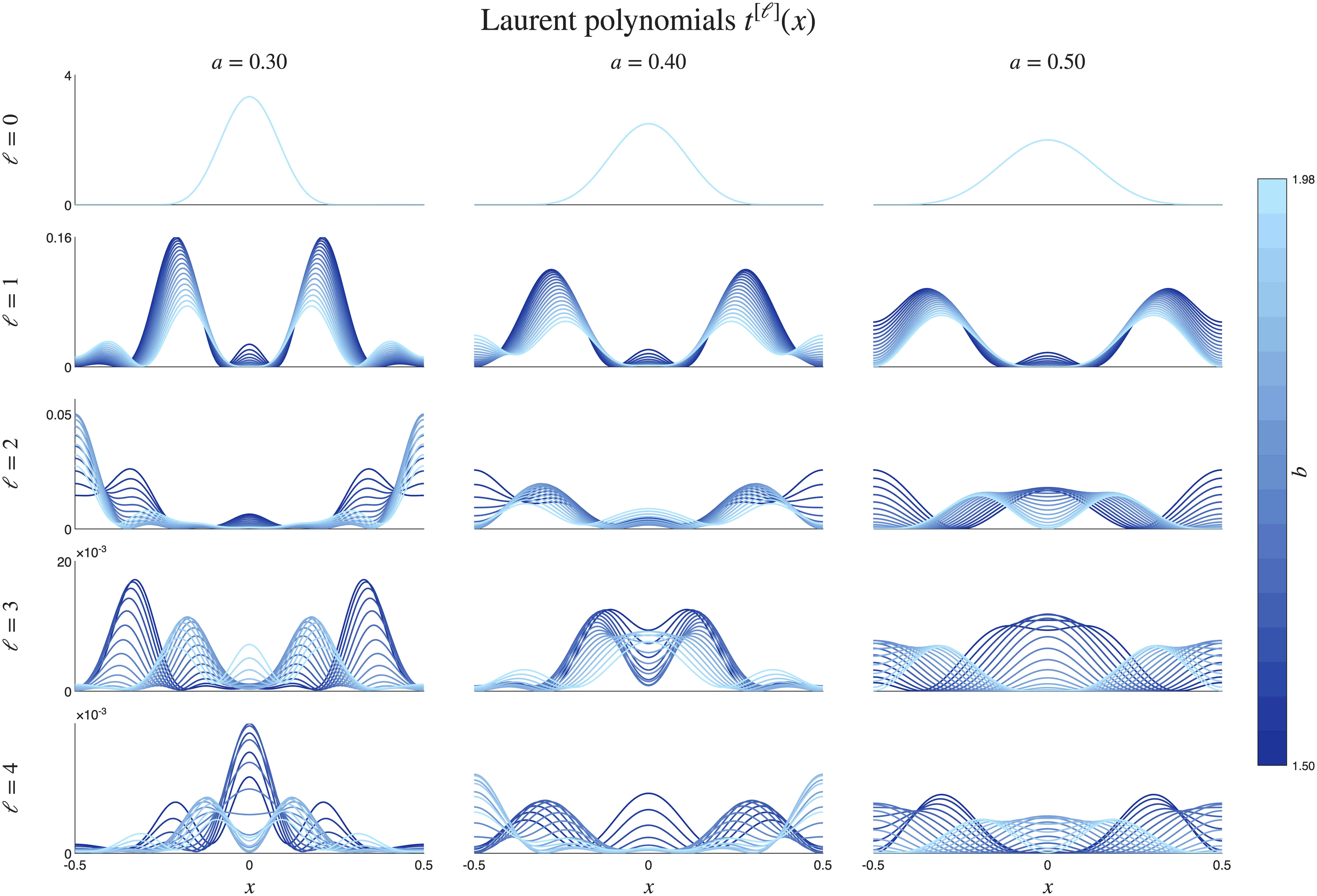}
         \caption{The Laurent polynomials $t^{[\ell]}(x) \in W$ associated to the Toeplitz blocks $T_{\infty}^{[\ell]} \in \mathcal{B}(l^p(\mathbb{Z}))$, $\ell=0, \hdots, 4$. The spectrum of $T_{\infty}^{[\ell]}$ is the interval $[\min t^{[\ell]}(x), \max t^{[\ell]}(x)]$. Here, $a=.3, .4, .5$ and the shades of blue correspond to different values of $b\in [1.5, 1.98]$.}
        \label{fig:laurent_poly}
  
\end{figure}

\begin{theorem} \label{thm:tl_sinc}
Let $N\geq2$. The Laurent polynomials $t^{[\ell]}(x)$ generated by the $N$th order
  B-spline $g=s_N$ have the form
  \begin{equation}
  t^{[\ell]}(x) = \frac{1}{a} \sum_{r \in \mathbb{Z}}
    \sinc^N\!\left(\frac{r - (x - ab\ell/2)}{a}\right)
	\sinc^N\!\left(\frac{r - (x + ab\ell/2)}{a}\right)
	\label{eq:tl_sinc}
  \end{equation}
\end{theorem}

  \begin{proof}  

Applying the change of variables \(y = x - \tfrac{a k}{2}\), so \(x = y + \tfrac{a k}{2}\) to \eqref{eq:tjl}, we obtain
\begin{equation*}
t^{[\ell]}_k 
= e^{\pi i a b \ell k}
  \int_{\mathbb{R}} g(y)\, g(y + a k)\, e^{2\pi i b \ell y}\,dy.
\end{equation*}
Define \(f_\ell(y) = g(y)\,e^{2\pi i b \ell y}\).
Then
\begin{equation*}
t^{[\ell]}_k 
= e^{\pi i a b \ell k}\, (f_\ell * g)(a k),
\qquad 
(f_\ell * g)(z) = \int g(y)\, g(y+z)\, e^{2\pi i b \ell y}\,dy.
\end{equation*}

Hence the Laurent polynomial has the form
\begin{equation*}
t^{[\ell]}(x)
= \sum_{|k|\leq \lfloor \frac{N}{a} \rfloor} e^{\pi i a b \ell k}\, (f_\ell * g)(a k)\, e^{2\pi i k x},
\end{equation*}
and applying the Poisson summation formula gives
\begin{equation*}
t^{[\ell]}(x)
= \frac{1}{a}\sum_{r\in\mathbb{Z}}
  \widehat{(f_\ell * g)}\!\left(\frac{r - (x - a b \ell/2)}{a}\right)= \frac{1}{a}\sum_{r\in\mathbb{Z}}
   \widehat{g}\!\left(\tfrac{r - (x - a b \ell/2)}{a}\right)
   \widehat{g}\!\left(\tfrac{r - (x - a b \ell/2)}{a} - b\ell\right).
\end{equation*}
The Fourier transform of $g=s_N$ has the form
\begin{equation*}
\widehat{g}(\xi)
= \sinc^N(\xi)
= \left(\frac{\sin(\pi \xi)}{\pi \xi}\right)^N.
\end{equation*}
Hence the representation of $t^{[\ell]}(x)$ is given by
\eqref{eq:tl_sinc}. 

\end{proof}

%

\begin{corollary} \label{corr: laurent_max_min}
    Suppose that $a = \tfrac{1}{p}$ with $p\in\mathbb{N}$,  $p\ge 2$, and $\ell\ge 0$ and let
$T_{\infty}^{[\ell]}$ denote the Toeplitz operator on $\ell^2(\mathbb{Z})$ with symbol
$t^{[\ell]}(x)$. Then
\begin{equation*}
\Spec\big(T_{\infty}^{[\ell]}\big) \subset [0,p].
\end{equation*}
Moreover, when $\ell=0$ we have 

\begin{equation*}
\Spec\big(T_{\infty}^{[0]}\big)= [0,p].
\end{equation*}

    \end{corollary}

    \begin{proof}
By Theorem~\ref{thm:tl_sinc} with $\ell=0$ and $a = 1/p$, we have
\begin{equation*}
t^{[0]}(x)
= \frac{1}{a}\sum_{r\in\mathbb{Z}} \sinc^{2N}\!\left(\frac{r-x}{a}\right)
= p \sum_{r\in\mathbb{Z}} \sinc^{2N}\big(p(r-x)\big).
\end{equation*}
At integer points $x\in\mathbb{Z}$, only the term $r=x$ contributes, so $t^{[0]}(x)=p$.

On the other hand, if $x = k/p$ with $k\in\mathbb{Z}\setminus p\mathbb{Z}$, then
$p(r-x) = pr-k$ is a nonzero integer for every $r\in\mathbb{Z}$, and hence
each $\sinc$ factor vanishes. Thus $t^{[0]}(k/p)=0$ for
$k\in\mathbb{Z}\setminus p\mathbb{Z}$.
Since $t^{[0]}$ is continuous and nonnegative, it follows that by Theorem \ref{thm:Tk_spec}
\begin{equation*}
\Spec(T_{\infty}^{[0]})=\left[\inf_x t^{[0]}(x), \sup_x t^{[0]}(x)\right] = [0,p].
\end{equation*}

\end{proof}

The numerical results in Figure \ref{fig:Tijdecay} demonstrate $N$th order decay of the spectral width $|\sup(t^{[\ell]}(x))-\inf(t^{[\ell]}(x))|$ of the finite Gram matrices $G_n$ for Gabor systems generated by the $N$th order B-spline. This gives a notion of block diagonal dominance of the matrix $G_n$.
\begin{figure}[h!]
	\centering
    \includegraphics[width=\linewidth]{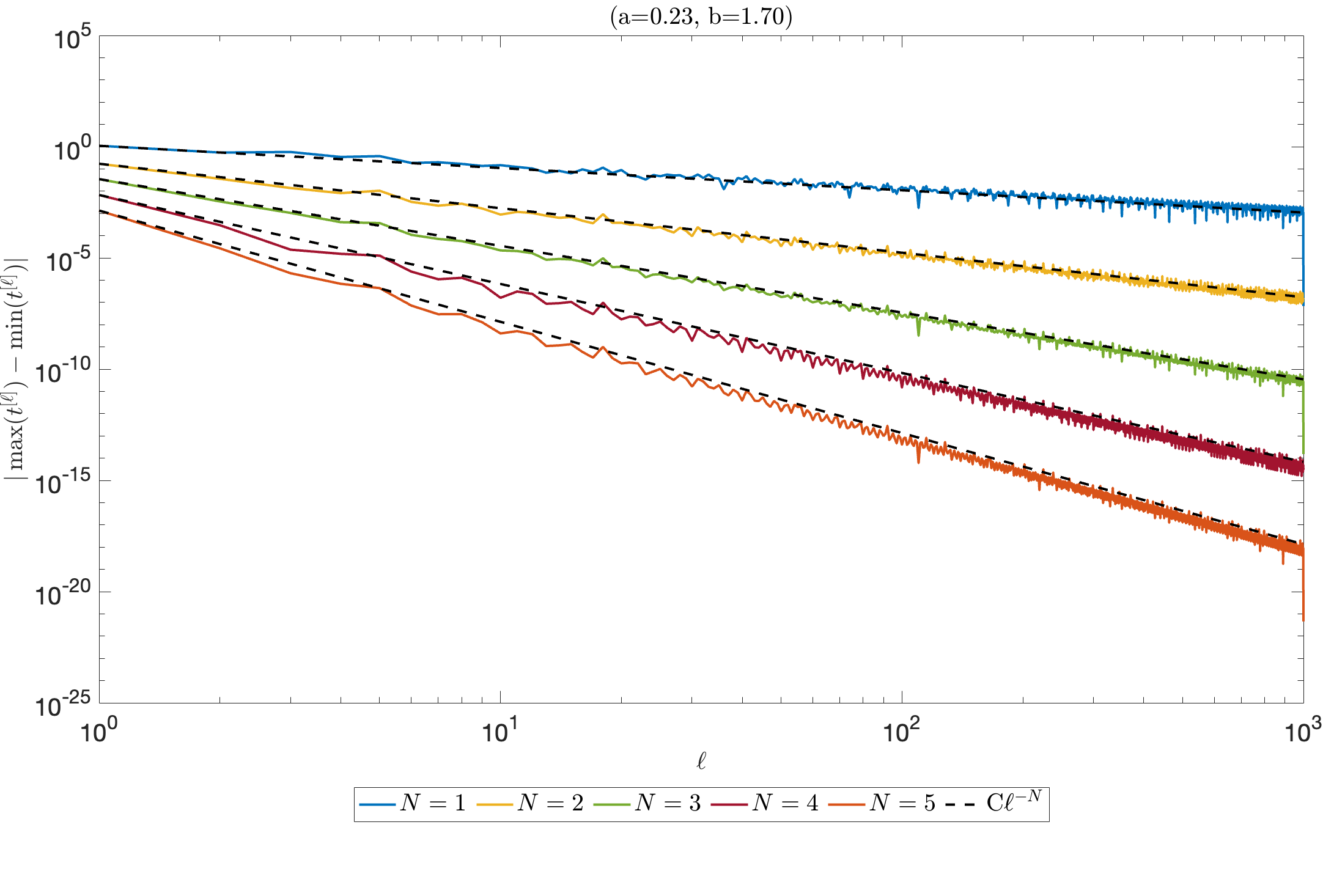}
	\caption{Width of spectrum of $t^{[\ell]}(x)$ for $a=.23$, $b=1.7$ corresponding to the $N$th order B-spline. The dashed lines show the corresponding $N$th order decay proved in Lemma \ref{lem:decay}.}\label{fig:Tijdecay}
\end{figure}

\begin{lemma}\label{lem:decay}
    Let $|\ell|>0$ and let $t^{[\ell]}(x)$ be the Laurent polynomial corresponding to the $N$th-order B-spline. Then, there exists a $C>0$ such that
    \[|\sup(t^{[\ell]}(x))-\inf(t^{[\ell]}(x))|\leq C {\ell}^{-N}.\]
\end{lemma}

\begin{proof} Suppose that $|\ell|>\tfrac{4}{ab}$ and let $x\in [-1/2, 1/2]$.  From~\eqref{eq:tl_sinc} we have:

\begin{align*}
|t^{[\ell]}(x)|& \leq \tfrac{a^{2N-1}}{\pi^{2N}} \tfrac{1}{|x-ab\ell/2|^N |x+ ab\ell/2|^N} + \tfrac{a^{2N-1}}{\pi^{2N}} \sum_{r=1}^{\infty} \tfrac{1}{|r-(x-ab\ell/2)|^N |r-(x+ab\ell/2)|^N} + \tfrac{a^{2N-1}}{\pi^{2N}} \sum_{r=-\infty}^{-1} \tfrac{1}{|r-(x-ab\ell/2)|^N |r-(x+ab\ell/2)|^N}\\
& = \tfrac{a^{2N-1}}{\pi^{2N}} \tfrac{1}{|x-ab\ell/2|^N |x+ ab\ell/2|^N} + \tfrac{a^{2N-1}}{\pi^{2N}} \sum_{r=1}^{\infty} \tfrac{1}{|r-(x-ab\ell/2)|^N |r-(x+ab\ell/2)|^N} + \tfrac{a^{2N-1}}{\pi^{2N}} \sum_{r=1}^{\infty} \tfrac{1}{|r+(x-ab\ell/2)|^N |r+(x+ab\ell/2)|^N}.
\end{align*}

We assume that $\ell>\tfrac{4}{ab}$ and the case $\ell<-\tfrac{4}{ab}$ is treated similarly.  

It follows that $$\sup_{x\in [-1/2, 1/2]}\tfrac{1}{|x-ab\ell/2|^N |x+ ab\ell/2|^N} \leq 4^{3N}(ab)^{-2N}\  \ell^{-2N}. $$ 

Let $L_0=\lfloor \tfrac{ab\ell +1}{2} \rfloor$. Then,

\begin{align*}
\sum_{r=1}^{\infty} \tfrac{1}{|r-(x-ab\ell/2)|^N |r-(x+ab\ell/2)|^N} &= \sum_{r=1}^{L_0-1} \tfrac{1}{|r-(x-ab\ell/2)|^N |r-(x+ab\ell/2)|^N} + \sum_{r=L_0}^{\infty} \tfrac{1}{|r-(x-ab\ell/2)|^N |r-(x+ab\ell/2)|^N}.
\end{align*}

By writing $\tfrac{ab\ell +1}{2}=L_0+\theta$
for some $\theta \in (0, 1)$ we see that 
\begin{align*}
\sum_{r=L_0}^{\infty} \tfrac{1}{|r-(x-ab\ell/2)|^N |r-(x+ab\ell/2)|^N} &\leq \sum_{r=L_0}^\infty \tfrac{1}{|r+\tfrac{ab\ell-1}{2}|^N} \tfrac{1}{|r-\tfrac{ab\ell +1}{2}|^N}\\
&\leq (L_0 + \tfrac{ab\ell-1}{2})^{-N} \sum_{r=L_0}^\infty \tfrac{1}{|r-\tfrac{ab\ell +1}{2}|^N}\\
&\leq 2^N(ab)^{-N}\ell^{-N} \sum_{r=L_0}^\infty \tfrac{1}{|r-L_0-\theta |^N}\\
&= 2^N(ab)^{-N}\ell^{-N} \sum_{r=0}^\infty \tfrac{1}{|r-\theta |^N}\leq_N \ell^{-N}.
\end{align*}

Next, for each $r\in \{1, 2, \hdots, L_0-1\}$ we have that $$\sup_{x\in [-1/2, 1/2]} \tfrac{1}{|r-(x-ab\ell/2)|^N} \leq \tfrac{2^N}{(1+ab\ell)^N}\leq 2^N(ab)^{-N}\ell^{-N}.$$

Moreover, $$\sup_{x\in [-1/2, 1/2]} \sum_{r=1}^{L_0-1} \tfrac{1}{ |r-(x+ab\ell/2)|^N}= \sup_{x\in [-1/2, 1/2]} \sum_{r=1}^{L_0-1} \tfrac{1}{ |r-x -L_0 -\theta +1/2|^N}= \sup_{x\in [0, 1]} \sum_{r=1}^{L_0-1} \tfrac{1}{ |r -L_0 +x -\theta |^N}<\infty $$ because $r-L_0-\theta \leq -1-\theta<-1$. 

We can estimate $\sum_{r=1}^{\infty} \tfrac{1}{|r+(x-ab\ell/2)|^N |r+(x+ab\ell/2)|^N} $ similarly to conclude that

$$|t^{[\ell]}(x)| \leq a^{2N-1}\pi^{-2N} (4^{3N}(ab)^{-2N}\ell^{-2N} +C_12^{N+1}(ab)^{-N}\ell^{-N}+C_22^{N+1}(ab)^{-N}\ell^{-N})\leq C_N \ell^{-N}$$ for some constants $C_1, C_2$, and $C_N$. 
    
\end{proof}

\subsection{Spectral analysis of the finite Gram matrix $G_n$}

Given Corollary~\ref{cor:infBlocksTn}, the infinite block $\mathscr{G}^{[\ell]}$  within $\mathscr{G}$  has the same singular values as the real-valued, banded Toeplitz matrix $T_{\infty}^{[\ell]}$ whose entries are given by \eqref{eq:tjl}.

Since $G_n^{[0]}$ is a principal sub-matrix of $G_n$, Cauchy's interlacing theorem states that their eigenvalues interlace, leading to bounds  on the minimum and maximum eigenvalues $\lambda_1(G_n)$ and  $\lambda_{n^2}(G_n)$ of ${G}_n$: 
%
    \begin{align*}
        \lambda_k(G_n) \leq \lambda_k(G_n^{[0]}) \leq \lambda_{k+n^2-n}(G_n), \qquad k=1,\hdots n.
    \end{align*}

           In particular,  $\lambda_1(G_n) \leq \lambda_1(G_n^{[0]})$ and $\lambda_n(G_n^{[0]}) \leq  \lambda_{n^2}(G_n)$. Figure \ref{fig:minmax_Gn_eigns} demonstrates this for $n=15$.

Consequently, we have the following result.

\begin{corollary} \label{corr: eigs}
   For $n\geq 1$, there exists an $K \in \mathbb{N}$ such that $\forall n \geq K$, we have the upper and lower estimates:
    \begin{align}
        & \lambda_1(G_n) \leq \inf_x t^{[0]}(x) \leq  \lambda_1(G_n^{[0]}) \\
        & \lambda_n(G_n^{[0]}) \leq \sup_x t^{[0]}(x) \leq \lambda_n(G_n).
    \end{align}
\end{corollary}

\begin{proof}
We will only establish the first set of inequalities as the second one is proved in a similar manner. 

We note that $\inf_x t^{[0]}(x) \leq  \lambda_1(T_n^{[0]}) $ follows from Corollary~\ref{cor-bdsfiniteteop}. Next, by the interlacing relation, we know that  $\lambda_1(G_n)\leq \lambda_1(G_n^{[0]})$. It follows that $$\limsup_{n}\lambda_1(G_n)\leq \limsup_{n} \lambda_1(G_n^{[0]})= \lim_{n} \lambda_1(G_n^{[0]})=\inf_x t^{[0]}(x). $$ Thus, there exists an integer $K\geq 1$ such that for all $n\geq K$, we have 
$$\lambda_1(G_n)\leq \inf_x t^{[0]}(x).$$

\end{proof}

\begin{figure}
    \centering
    \includegraphics[width=\linewidth]{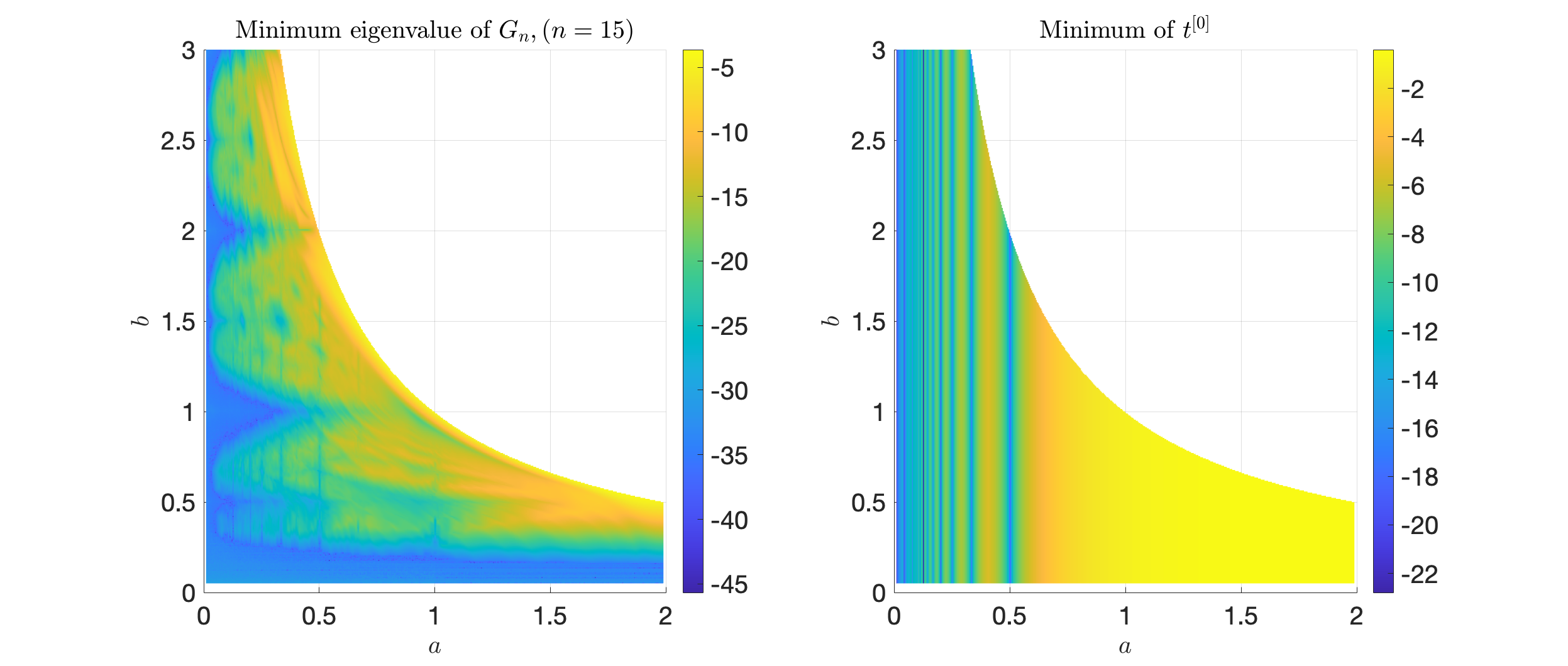}
    \includegraphics[width=\linewidth]{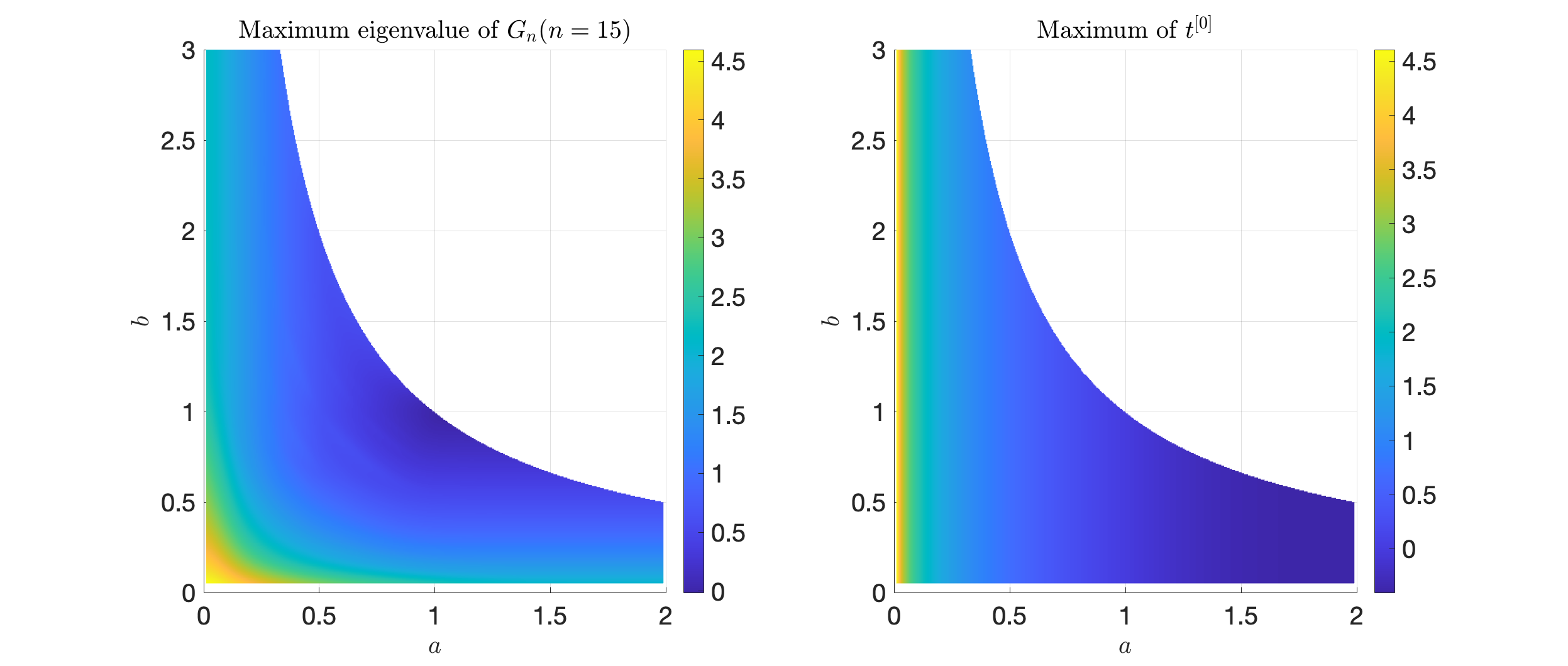}
    \caption{Plots of $\lambda_1(G_n)$ (top left), $\min t^{[0]}(x)$ (top right), $\lambda_n(G_n)$ (bottom left), and $\max t^{[0]}(x)$ (bottom right) for $n=15$ demonstrate the results of Corollary \ref{corr: eigs}.}
    \label{fig:minmax_Gn_eigns}
    
\end{figure}

\subsection{Asymptotically Equivalent Circulant Matrices}\label{sec:circulant}

The exact computation of the spectrum of a Toeplitz matrix \( G_n \), especially for large \( n \), is generally intractable. However, Theorem~\ref{thm:offdiagTH} shows that the spectral behavior of \( G_n \) is governed by the singular values of the Toeplitz blocks \( T_n^{[\ell]} \). To analyze these efficiently, we introduce a family of \emph{circulant matrices} \( C_n^{[\ell]} \) that approximate the Toeplitz matrices \( T_n^{[\ell]} \) in a precise asymptotic sense.

\subsubsection{Background on Asymptotic Equivalence}


Let \( T_n \) be a banded Toeplitz matrix of bandwidth \( m \), defined by a symbol \( t(x) = \sum_{k=-m}^{m} t_k e^{2\pi i k x} \). Define the corresponding circulant matrix \( C_n \in \mathbb{C}^{n \times n} \) by wrapping the Toeplitz coefficients as:
\[
c_k = 
\begin{cases}
t_{-k}, & 0 \leq k \leq m, \\
t_{n-k}, & n-m \leq k \leq n-1, \\
0, & \text{otherwise}.
\end{cases}
\]
This defines the first row of \( C_n \), from which the rest is filled by circular shifts.

\begin{lemma}[Lemma 4.2 in \cite{gray2006toeplitz}]\label{lem:circeigconv} The matrices $T_n$ and $C_n$ are asymptotically equivalent, denoted by $T_n\sim C_n$ i.e.,
\begin{enumerate}
    \item Both sequences are uniformly bounded in operator norm: \( \|T_n\|_2, \|C_n\|_2 \leq M \) for some constant \( M \) independent of \( n \),
    \item The difference vanishes in the normalized Hilbert–Schmidt norm:
    \[
    \lim_{n \to \infty} \left( \frac{1}{n} \sum_{i,j=1}^n |(T_n)_{i,j} - (C_n)_{i,j}|^2 \right)^{1/2} = 0.
    \]
\end{enumerate}
    
\end{lemma}

When $T_n\sim C_n$ holds for Hermitian matrices, their eigenvalues are said to be \emph{asymptotically equally distributed}, meaning that for any integer \( s \geq 1 \),
\[
\lim_{n \to \infty} \frac{1}{n} \sum_{m=1}^{n} \left( \lambda^s_m(A_n) - \lambda^s_m(C_n) \right) = 0.
\]
For general theory, see \cite[Chapter 2]{gray2006toeplitz}. In special cases, such as banded Toeplitz matrices, stronger pointwise convergence of eigenvalues is known \cite{zhu2017asymptotic}.

Circulant matrices admit explicit spectral decompositions. Specifically:

\begin{theorem}[Theorem 5.15, \cite{bottcher2000toeplitz}]\label{thm:C_eigs}
Let \( C_n \) be the circulant matrix constructed from a trigonometric polynomial \( t \). Then its eigenvalues are given (up to permutation) by uniformly spaced samples of \( t \):
\[
\operatorname{spec}(C_n) = \left\{ t\left( \tfrac{p}{n} \right) : p = 0, \dots, n-1 \right\}.
\]
\end{theorem}

\subsubsection{Application to Gram Matrices for $\mathcal{G}(s_N,\Lambda_n)$}

Let \( T_n^{[\ell]} \) denote the Toeplitz blocks with symbol $t^{\ell}(x)$ generated by the second order B-spline, and define \( C_n^{[\ell]} \) to be the circulant matrix associated with the same symbol. 

Therefore, letting \( \{ \psi_1, \dots, \psi_n \} \) be the ordered eigenvalues of \( C_n^{[\ell]} \), i.e.,
\[
\psi_m = t^{[\ell]}\left( \tfrac{p_m}{n} \right)
\quad \text{for } p_m \in \{0, \dots, n-1\},
\]
we have that for any \( s \geq 1 \),
\[
\lim_{n \to \infty} \frac{1}{n} \sum_{m=1}^{n} \left( \lambda^s_m(T_n^{[\ell]}) - \psi^s_m \right) = 0.
\]
This approximation enables efficient estimation of spectral bounds for the finite Gram matrix \( G_n \). The asymptotic behavior of the difference in eigenvalues between $C_n^{[\ell]}$ and $T_n^{[\ell]}$ is shown in Figure \ref{fig: eig_diff}.

\begin{figure}[h!]
    \centering
    \includegraphics[width=0.9\textwidth]{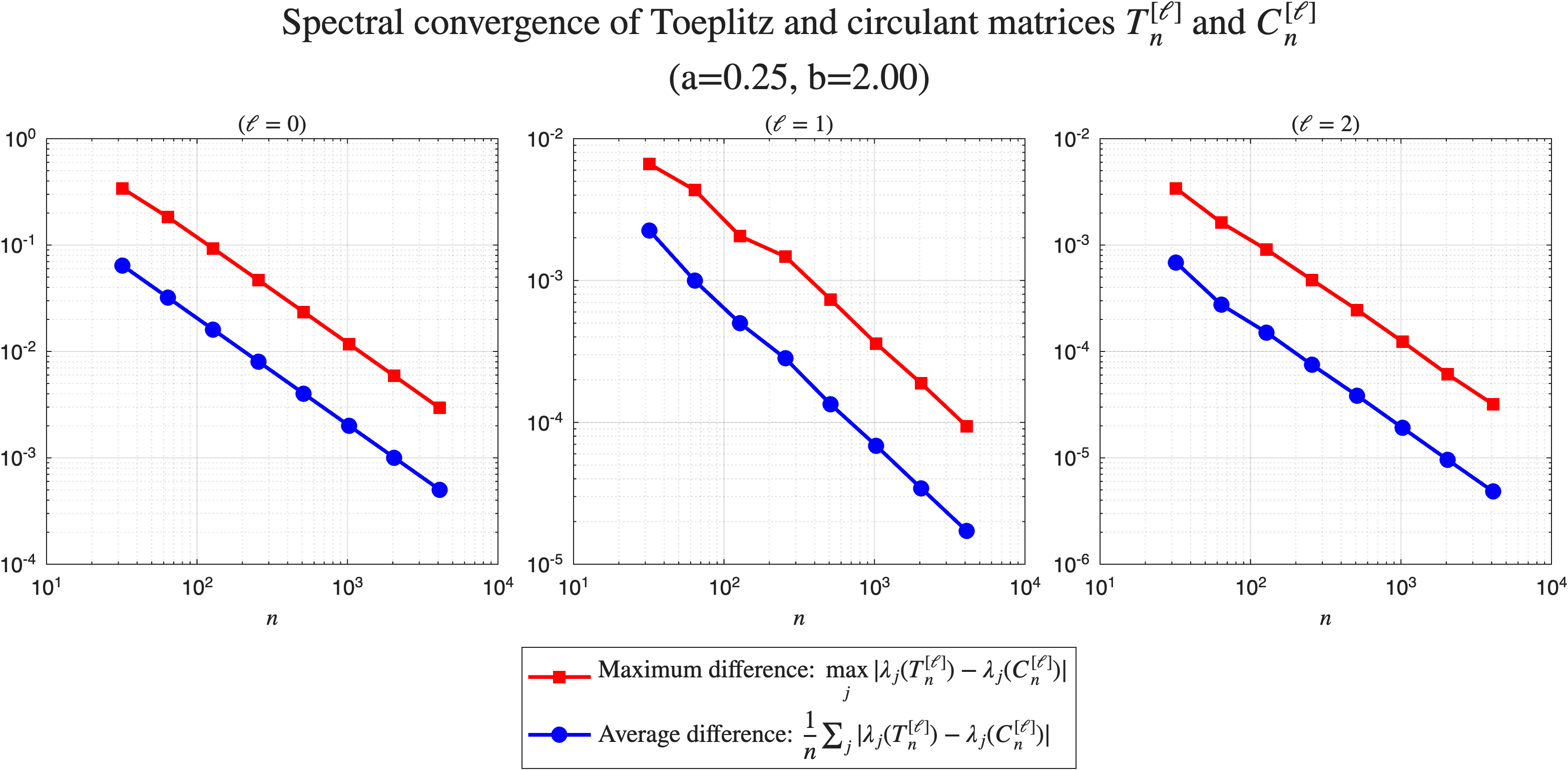}
    \caption{Lemma \ref{lem:circeigconv} gives the convergence of eigenvalues of the Toeplitz block \( T_n^{[\ell]} \) to those of the circulant matrix \( C_n^{[\ell]} \). Here, $T_n^{[\ell]}$ correspond to the subblocks $G_n^{[\ell]}$ of the Gram matrices $G_n$ for $\mathcal{G}(s_2,\Lambda_n)$.}
    \label{fig: eig_diff}
\end{figure}

\section{Conclusions}\label{sec:conclusion}
We conclude by connecting the results of this paper to the frame set problem for the B-spline.

We recall that if $\mathscr{G}$ is the Gram operator for a frame with frame bounds $\alpha, \beta$, then
\[\{0\}\subseteq \sigma(\mathscr{G})\subseteq\{0\}\cup [\alpha, \beta]. \]
Let  $\mathscr{G}^\dagger$ denote its Moore-Penrose pseudoinverse.
Estimates on the operator norm of the pseudo-inverse $\|\mathscr{G}^\dagger\|$ are also valuable as its reciprocal is the lower frame bound $\alpha = \|\mathscr{G}^\dagger\|^{-1}$, see \cite{adcock2019frames}.

Furthermore, the following result can be applied to obtain frame bounds.

\begin{lemma}[Lemmas 4 and 5 in \cite{adcock2019frames}]\label{lem:adcock}
Let  $\Phi$ be a linearly independent frame with frame bounds $\alpha$ and $\beta$.
 \begin{enumerate}
 \item The truncated Gram matrix $G_n$ is invertible with $\|G^{-1}_n\|^{-1}=A_n$ and  $\|G_n\|=B_n$, where $A_n$ and $B_n$ are the frame bounds of the truncated frame $\Phi_n$.
 \item The sequences $\{A_n\}_{n\in \mathbb{Z}}$ and $\{B_n\}_{n\in \mathbb{Z}}$ are monotonically nonincreasing and monotonically nondecreasing, respectively.
 \item $B_n\leq \beta$ for all $n$ and $\lim_{n\to\infty}B_n=\beta$.
 \item $\liminf_{n}A_n>0$ iff $\Phi$ is a Riesz basis.
\end{enumerate}

\end{lemma}

\begin{remark}
\begin{enumerate}
\item Although we do not know if $\mathcal{G}(s_N,a\mathbb{Z}\times b\mathbb{Z})$ is a frame, we know that any finite section is linearly independent. This is a special case of the HRT-Conjecture which asserts that any finite collection of time-frequency shifts of a non-zero $L^2$ function is linearly independent \cite{heil1996linear, heil2007history}. Given that the HRT-conjecture is true for all compactly supported $L^2$ functions, we can prove parts 1--3 of Lemma~\ref{lem:adcock} for the Gabor systems generated by $s_N$ and all parameters $0<ab<1.$ Furthermore, we know that by the Balian-Low Theorem, $\mathcal{G}(s_N, a\mathbb{Z}\times b\mathbb{Z})$ is never a Riesz basis. As such we can conclude that $\liminf A_n=0$
\item Our results  provide upper and lower estimates on the frame bounds associated with the finite Gabor frame $\mathcal{G}(s_N,\Lambda_n)$.
\item The Bessel bound of the Gabor
system $\mathcal{G}(s_N,a\mathbb{Z}\times b\mathbb{Z})$ is the supremum of the spectrum of $T_{\infty}^{[0]}$, i.e., $B = \sup_x t^{[0]}(x)$. Suppose that $a = \tfrac{1}{p}$ with $p\in\mathbb{N}$ and $p\ge 2$. Then, Corollary \ref{corr: laurent_max_min} implies that  the Bessel bound of the Gabor system
$\mathcal{G}(g,\Lambda)$ is $\beta = p$. 
\item If $a \leq 1/2$, then the scaled and shifted $\text{sinc}^2(\frac{1}{a}(\cdot-l))$ share zeros with $\text{sinc}^2(\cdot)$ at the scaled integers $\frac{1}{a}\mathbb{Z}$. Hence, whenever $m/n$ coincides with a zero of $t^{[0]}(x)$ (for example, when $a n$ is an integer and $m$ is a multiple of $a n$), we have $\lambda_m(C_n^{[0]}) = 0$. Because the eigenvalues of the finite blocks $T_n^{[0]}$ interlace those of ${G}_n$, it must be the case that if the minimum eigenvalue of $T_n^{[0]}$ tends to zero, so does that of ${G}_n$. The minimum eigenvalue of $T_n^{[0]}$ tends to zero when $C_n^{[0]}$ has an eigenvalue of zero or tends to zero. This corresponds to when $t^{[0]}(x)$ has a zero, see Figure \ref{fig:laurent_poly}. If $a=\frac{1}{4}$ and $n=1024$, then 
\begin{equation} \label{eq: circ_eig}
    \textrm{spec}(C_n^{[0]})  = \left\{4 \sum_{r=-\infty}^{\infty} \text{sinc}^4 (\frac{p}{256}-4r )\right\}_{p=0}^{n-1}
\end{equation}
has zeros at $p=256, 512, 768$. Hence, when $a=\frac{1}{4}$ it must be that the eigenvalue and frame bound $\alpha(a,b,n)=\lambda_1(G_n)\rightarrow 0$ as $n\to \infty$.

\end{enumerate}
 
\end{remark}

It is known that for truncations of linearly independent frames, the lower frame bound must tend to zero \cite{adcock2019frames}. Thus, the following corollary can be viewed as generalizing this result to Gabor systems in general and tying them to the eigenvalues of their Gram matrices.


\begin{corollary} (Lower Frame Bound and Eigenvalue Decay) \label{cor: min_eig_finite_gabor}
    If the minimum eigenvalue $\lambda_1(G_n^{[0]})$ of the of the finite block $G_n^{[0]}\in \mathbb{C}^{n\times n}$ tends to zero as $n \rightarrow\infty$, then (1) $a \in \mathbb{Q}$ and (2) the minimum eigenvalue $\lambda_1(G_n)$ of the finite Gram matrices ${G}_n\in \mathbb{C}^{n^2\times n^2}$ tends to zero as $n \rightarrow \infty$.
\end{corollary}




In this paper, we have provided a detailed analysis of the structural properties of the Gram matrix $\mathscr{G}$ associated with Gabor systems generated by B-splines of order $N$. By exploiting the compact support of $s_N$, we established a structural hierarchy that allows for a rigorous treatment of the properties of the frame operator's finite sections.

Our primary structural result, Theorem \ref{thm:offdiagTH}, demonstrates that the blocks of the Gram matrix $G_n^{[\ell]}$ admit a Hadamard factorization $G_n^{[\ell]} = T_n^{[\ell]} \circ H_n^{[\ell]}$. This decomposition isolates the modulation-induced phase into a rank-one Hankel matrix $H_n^{[\ell]}$, while the singular values are dictated by a real-valued, symmetric, and banded Toeplitz matrix $T_n^{[\ell]}$. This result connects Gabor analysis and classical Toeplitz theory, enabling the use of Szegő-type limit theorems.


\section{Acknowledgements}
Martin Buck and Kasso Okoudjou were partially supported by grants from the National Science Foundation under grant numbers DMS 2205771, and DMS 2309652. Christina Frederick and Alexander Stangl were supported under NSF grant DMS 2309651.

The authors also acknowledge helpful discussions with Prof. Hans Feichtinger.

The authors used AI tools to edit and polish the manuscript for spelling, grammar, and
general style. The final content and scientific conclusions remain the responsibility of the
authors
\bibliography{bib.bib}
\bibliographystyle{abbrv}

\end{document}